\documentclass{amsart}

\usepackage{amssymb}

\newtheorem{thrm}{Theorem}[section]
\newtheorem{lemma}[thrm]{Lemma}
\newtheorem{prop}[thrm]{Proposition}
\newtheorem{cor}[thrm]{Corollary}

\theoremstyle{definition}
\newtheorem{defn}[thrm]{Definition}

\theoremstyle{remark}

\numberwithin{equation}{section}
\newcommand{\dbar}{$\bar{\partial}$}
\newcommand{\mdbar}{\bar{\partial}}

\newcommand{\lre}{\mathcal{E}}
\newcommand{\lra}{\mathcal{A}}
\newcommand{\lrs}{\mathcal{S}}

\newcommand{\lrl}{\mathcal{L}}

\newcommand{\lrt}{\mathcal{T}}
\newcommand{\lrp}{\mathcal{P}}
\newcommand{\Z}{\mathcal{Z}}

\begin{document}

\title[The Bergman projection and estimates]
{The Bergman projection and weighted $C^k$ estimates
 for the canonical solution to \dbar\ on non-smooth domains}


\author{Dariush Ehsani}
\address{Department of Mathematics, Penn State - Lehigh Valley, Fogelsville,
PA 18051}
 \curraddr{Humboldt-Universit\"{a}t, Institut f\"{u}r Mathematik,
10099 Berlin}
 \email{ehsani@psu.edu}
\thanks{Partially supported by the Alexander von Humboldt Stiftung}



\date{}

\dedicatory{}

\begin{abstract}
 We apply integral representations for
functions on non-smooth strictly pseudoconvex domains, the
Henkin-Leiterer domains, to derive weighted $C^k$ estimates for
the component of a given function, $f$, which is orthogonal to
holomorphic functions in terms of $C^k$ norms of $\mdbar f$. The
weights are powers of the gradient of the defining function of the
domain.
\end{abstract}

\maketitle

\section{Introduction}

Let $X$ be an $n$-dimensional complex manifold, equipped with a
Hermitian metric, and $D\subset\subset X$ a strictly pseudoconvex
domain with defining function $r$.  We allow for singularities in
the boundary, $\partial D$ of $D$ by permitting the possibility
that $dr$ vanishes at points on $\partial D$.  Such domains were
first studied by Henkin and Leiterer in \cite{HL}, and we
therefore refer to them as Henkin-Leiterer domains.

  We shall make the additional assumtion that $r$ is a Morse
function.  Let $U$ be a neighborhood of $\partial D$.  Then
\begin{equation*}
U\cap D=\{ x\in U: r(x)<0\},
\end{equation*}
 $r$ with only non-degenerate critical points on $U$.  We have
\begin{equation*}
\partial D=\{ x:r(x)=0\},
\end{equation*}
and we can assume that there are finitely many critical points on
$bD$, and none on $U\setminus bD$.

In \cite{EhLi}, Lieb and the author studied the Bergman projection
on Henkin-Leiterer domains in $\mathbb{C}^n$, and obtained
weighted $L^p$ estimates.
 Set $\gamma(\zeta)=|\partial r(\zeta)|$, and let us define the
 weighted $L^p$ spaces by
 \begin{equation*}
 L^{p,k}(D)=\left\{ f: \|f\|_{p,k}^p=\int_D\gamma^{k}
  |f|^p dV <\infty \right\}.
\end{equation*}
\begin{thrm} [Ehsani-Lieb]
\label{main intro}
 Let $D\subset\subset\mathbb{C}^n$ be a Henkin-Leiterer domain
 with a Morse defining function.
The Bergman projection is continuous from $L^p(D)$ into
$L^{p,k}(D)$ for $2\le p<\infty$ and $k=k(p)$ sufficiently large.
$k$ depends on $p$ (with $k=0$ for $p=2$).
\end{thrm}
Our purpose in this paper is to extend the results of \cite{EhLi}.
In one direction we look to extend the results into the setting of
Henkin-Leiterer domains in complex manifolds.  The integral
representation used in \cite{EhLi} relied on integral operators
constructed with the use of the Henkin-Ram\'{i}rez function, and
is not readily adaptable to the manifold setting.  In this paper
we use integral operators constructed from the Levi polynomial and
the geometric distance function as in \cite{LR86}.  One of our
main results is the reproduction of the weighted $L^p$ estimates
of \cite{EhLi} when $D$ lies in a complex manifold.

We further establish $C^k$ estimates in the theorem
\begin{thrm}
\label{ckintro}
 Let $D\subset\subset X$ be a Henkin-Leiterer
domain with a Morse defining function.  Let $f\in L^2(D)$ and $B$
denote the Bergman projection. Then for $\epsilon>0$, we have
\begin{equation*}
\|\gamma^{(n+2)+\epsilon+k} (f-Bf)\|_{C^{k}}\lesssim
\|\gamma^k\mdbar f\|_{C^k} +\|f\|_{L^2}.
\end{equation*}
\end{thrm}

One of the obvious difficulties with which one comes to face in
trying to establish $C^k$ estimates on non-smooth domains is the
choice of frame of vector fields with which one will work.  In the
case of smooth domains a special boundary chart is used in which
$\omega^n=\partial r$ is part of an orthonormal frame of
$(1,0)$-forms.  When $\partial r$ is allowed to vanish, the frame
needs to be modified.  We get around this difficulty by defining a
$(1,0)$-form, $\omega^n$ by $\partial r = \gamma \omega^n$.  In
the dual frame of vector fields we are then faced with factors of
$\gamma$ in the expressions of the vector fields with respect to
local coordinates, and we deal with these terms by multiplying our
vector fields by a factor of $\gamma$.  This ensures that when
vector fields are commuted, there are no error terms which blow up
at the singularity.

The author wishes to thank Ingo Lieb by whom the problem of the
Bergman projection on Henkin-Leiterer domains was originally
suggested.  Many of ideas in the current paper were discussed with
Ingo Lieb during the author's stay at the Max Plank Institute for
mathematics in Bonn in 2005, and many helpful comments were taken
into account.

\section{Notation}

With local coordinates denoted by $\zeta_1,\ldots,\zeta_n$, we
equip $X$ with a Hermitian metric
\begin{equation*}
ds^2= \sum_{j,k}  g_{jk}(\zeta)d\zeta_j d\bar{\zeta}_k.
\end{equation*}

We take $\rho(x,y)$ to be a symmetric, smooth function on $X\times
X$ which coincides with the geodesic distance in a neighborhood of
the diagonal, $\Lambda$, and is positive outside of $\Lambda$.

 We shall write
$\lre_{j}(x,y)$, for $j\ge 0$, for those double forms on open sets
$U\subset X\times X$ such that $\lre_{j}$ is smooth on $U$ and
satisfies
\begin{equation*}
\lre_{j}(x,y)\lesssim  \rho^j(x,y).
\end{equation*}

We follow \cite{LR86} to construct kernels to be used in our
integral representation.  Let $V^i\subset \subset U^i$ be two open
coverings of the boundary $\partial D$.  $\delta$ and
$\varepsilon$ ($ \varepsilon\le \delta$) are chosen such that
\begin{align*}
&i)\ V^i \mbox{ cover the set } \{ r(x)\le 3\delta\}\\
&ii)\ x\in V^i \mbox{ and } \rho(x,y)\le \varepsilon
 \mbox{ imply } y\in U^i\\
&iii)\ |r(x)|\le\delta \mbox{ and } \rho(x,y)\le \varepsilon
 \mbox{ imply } |r(y)|< 2\delta\\
&iv)\ |r(x)|\le 2\delta \mbox{ and } \rho(x,y)\le \varepsilon
 \mbox{ imply } |r(y)|< 3\delta.
\end{align*}
 Let
$(\zeta^i,z^i)$ be local coordinates on $U^i\times U^i$.  We
define the Levi polynomial $F^i$ on $U^i\times U^i$ by
\begin{equation*}
F^i(\zeta,z)= \sum_{j=1}^n\frac{\partial
 r}{\partial\zeta_j}(\zeta)(\zeta_j-z_j)
  -\frac{1}{2}\sum_{j,k=1}^n\frac{\partial^2
  r}{\partial\zeta_j\partial\zeta_k}(\zeta_j-z_j)(\zeta_k-z_k),
\end{equation*}
where, for ease of notation, we drop the superscripts, $i$, on the
local coordinates.

We also choose a smooth partition of unity, $\xi^i$ such that
$\mbox{supp } \xi^i\subset\subset V^i$ and $\sum\xi^i=1$ on $\{
|r(x)|\le 3\delta\}$.  We set
\begin{equation*}
F(x,y)=\sum_i \xi^i(x)F^i(x,y).
\end{equation*}
We choose a smooth symmetric patching function on $X\times X$,
$\varphi$, such that $0\le \varphi(x,y)\le 1$ and
\begin{equation*}
\varphi(x,y)=
 \begin{cases}
1 & \mbox{for } \rho^2(x,y)\le
\frac{\varepsilon}{2}\\
0 &  \mbox{for } \rho^2(x,y)\ge \frac{3}{4}\varepsilon,
\end{cases}
\end{equation*}
and set
\begin{equation*}
\phi(x,y)
 =\varphi(x,y)(F(x,y)-r(x))
  +(1-\varphi(x,y))\rho^2(x,y).
 \end{equation*}

We write here the essential properties which we use in this paper.
Let $D_{\delta}=\{x: r(x)<\delta)\}$ and $S_{\delta}=\{x:
-\delta<r(x)<\delta)\}$.
\begin{lemma}
 \label{fundphi}
There exist positive constants $\epsilon_0$ and $\delta_0$ such
that on $\overline{S}_{\delta_0}\times\overline{D}_{\delta_0}$
\begin{enumerate}
\item[i)] $\phi=0   \mbox{ iff } x=y\in \partial D$
 \item[ii)]
$|\phi|\gtrsim 1 \mbox{ if } \rho(x,y)\ge\epsilon_0$
 \item[iii)] $2 \mbox{Re }\phi(x,y)\gtrsim
 -r(x)-r(y)+\rho^2(x,y)$
 \item[iv)] $\phi-\phi^{\ast}=\lre_3$.
\end{enumerate}
\end{lemma}


Since $D$ is strictly pseudoconvex and $r$ is a Morse function, we
can take $r_{\epsilon}=r+\epsilon$ for $\epsilon$ small enough,
and then $r_{\epsilon}$ will be defining functions for smooth,
strictly pseudoconvex $D_{\epsilon}$.

We define $\phi_{\epsilon}(x,y)$ as we did $\phi(x,y)$ above with
$r$ replaced by $r_{\epsilon}$.  We also define
\begin{align*}
 \label{pdefn}
 &   P(x,y)=\rho^2(x,y)+
2r(x)r(y)\\
&P_{\epsilon}(x,y)=\rho^2(x,y)+ 2r_{\epsilon}(x)r_{\epsilon}(y).
\end{align*}

For $N\ge 0$, we let $R_N$ denote an $N$-fold product, or a sum of
such products, of first order vector fields applied to $r(y)$,
with the notation $R_0=1$.

\begin{defn} A double differential form $\lra^{\epsilon}(x,y)$ on
$\overline{D}_{\epsilon}\times\overline{D}_{\epsilon}$ is an
\textit{admissible} kernel, if it has the following properties:
\begin{enumerate}
\item[i)] $\lra^{\epsilon}$ is smooth on
$\overline{D}_{\epsilon}\times\overline{D}_{\epsilon}-\Lambda_{\epsilon}$
 \item[ii)] For each point $(x_0,x_0)\in \Lambda_{\epsilon}$ there is
 a neighborhood $U\times U$ of $(x_0,x_0)$ on which $\lra^{\epsilon}$ or $\overline{\lra}^{\epsilon}$
 has the representation
 \begin{equation}
 \label{typerep}
  R_N \lre_{j}
  P^{-t_0}_{\epsilon}\phi^{t_1}_{\epsilon}\overline{\phi}^{t_2}_{\epsilon}
   \phi^{\ast t_3}_{\epsilon}\overline{\phi}^{\ast t_4}_{\epsilon} r^l_{\epsilon} r^{\ast
   m}_{\epsilon}
 \end{equation}
with $N,j, t_0, \ldots, m$ integers and $j, t_0, l, m \ge 0$,
 $-t=t_1+\cdots+t_4\le 0$.
\end{enumerate}
The above representation is of \textit{smooth type} $s$ for
\begin{equation*}
s=2n+j+\min\{2, t-l-m\} -2(t_0+t-l-m).
\end{equation*}
We define the \textit{type} of $\lra^{\epsilon}(x,y)$ to be
\begin{equation*}
\tau=s-\max \{ 0,2-N\}.
\end{equation*}
  $\lra^{\epsilon}$ has \textit{smooth
type} $\ge s$ if at each point $(x_0,x_0)$ there is a
representation (\ref{typerep}) of smooth type $\ge s$.
$\lra^{\epsilon}$ has \textit{type} $\ge \tau$ if at each point
$(x_0,x_0)$ there is a representation (\ref{typerep}) of type $\ge
\tau$.  We shall also refer to the \textit{double type} of an
operator $(\tau,s)$ if the operator is of
 type $\tau$ and of smooth type $s$.
\end{defn}
The definition of smooth type above is taken from \cite{LR86}.
Here and below $(r_{\epsilon}(x))^{\ast }=r_{\epsilon}(y)$, the
$\ast$ having a similar meaning for other functions of one
variable.

For $\lra_j^0$, we will simply write $\lra_j$. We also denote by
$A_j^{\epsilon}$ to be operators with kernels of the form
$\lra_j^{\epsilon}$.  $A_j$ will denote the operators with kernels
$\lra_j$.  We use the notation $\lra_{(j,k)}^{\epsilon}$ (resp.
$\lra_{(j,k)}$) to denote kernels of double type $(j,k)$.

\section{Properties of operators}
We collect in this section the various mapping properties of our
operators.  The proofs follow as in \cite{Eh09a, Eh09b}.

A non-vanishing vector field, $T$, will be called tangential if
$Tr=0$ on $r=0$.  Near a boundary point, we choose a coordinate
patch on which we have an orthogonal frame $\omega^1,\ldots,
\omega^n$ of $(1,0)$-forms with $\partial r=\gamma\omega^n$.  Let
$L_1,\ldots,L_n$ denote the dual frame. $L_1,\ldots, L_{n-1}$,
$\overline{L}_1,\ldots,\overline{L}_{n-1}$, and
$Y=L_n-\overline{L}_n$ are tangential vector fields.
$N=L_n+\overline{L}_n$ is a normal vector field.  We say a given
vector field $X$ is a smooth tangential vector field if it is a
tangential field and if near each boundary point $X$ is a
combination of such vector fields $L_1,\ldots, L_{n-1}$,
$\overline{L}_1,\ldots,\overline{L}_{n-1},Y$, and $rN$ with
coefficients in $C^{\infty}(\overline{D})$.

We define the function spaces with which we will be working.
\begin{defn}  Let $0\le\beta$ and $0\le\delta$.  We define
\begin{equation*}
\|f\|_{L^{\infty,\beta,\delta}(D)}=\sup_{x\in D}
 |f(x)|\gamma^{\beta}(x)|r(x)|^{\delta}.
\end{equation*}
\end{defn}
\begin{defn}
We set for $0<\alpha<1$
\begin{equation*}
\Lambda_{\alpha}(D)
 =\{ f\in L^{\infty}(D)\ |\ \|f\|_{\Lambda_{\alpha}}
 := \|f\|_{L^{\infty}}+\sup
 \frac{|f(x)-f(y)|}{|x-y|^{\alpha}}<\infty \}.
\end{equation*}
\end{defn}
We also define the spaces $\Lambda_{\alpha,\beta}$ by
\begin{equation*}
 \Lambda_{\alpha,\beta}:=\{ f: \|f\|_{\Lambda_{\alpha,\beta}}=
\|\gamma^{\beta}f\|_{\Lambda_{\alpha}}<\infty \}.
\end{equation*}

\begin{thrm}
\label{dertype1}
 Let $T$ be a smooth first order tangential differential operator on $D$.
 For $A$ an operator of type 1 we have
\begin{align*}
&i)\ A^{}:L^{\infty,2+\epsilon,0}(D)\rightarrow
 \Lambda_{\alpha,2-\epsilon'}(D)
\qquad 0<\epsilon,\epsilon', \quad \alpha+\epsilon+\epsilon'<1/4\\
&ii)\ A_j^{}:L^p(D)\rightarrow L^s(D) \quad
\frac{1}{s}>\frac{1}{p}-\frac{j}{2n+2}\\
 &iii)\ \gamma^{\ast}T A:\L^{\infty,2+\epsilon,0}(D) \rightarrow
 L^{\infty,\epsilon',\delta}(D) \qquad 1/2<\delta<1 \quad \epsilon<\epsilon'<1\\
 &iv)\ A:L^{\infty,\epsilon,\delta}(D)\rightarrow
 L^{\infty,\epsilon',0}(D) \qquad \epsilon<\epsilon'
 \quad \delta<1/2 +(\epsilon'-\epsilon)/2.
\end{align*}
\end{thrm}
\begin{thrm} Let $1< p< \infty$.  Then
 \label{a0a2}
\begin{equation*}
A_0:L^p(D)\rightarrow L^p(D)
\end{equation*}
\end{thrm}
\begin{proof}
The proof of Theorem 3.4 of \cite{EhLi} may be adapted to our
situation to prove the inequality
 \begin{equation*}
 \sup_{y\in\Omega}\int |\lra_0(x,y)|
 |r(x)|^{-\delta}|r(y)|^{\delta}dV(x) <\infty
\end{equation*}
for $\delta<1$, from which the theorem follows after an
application of H\"{o}lder's inequality.
\end{proof}

We let $\lre_{j-2n}^i(x,y)$, $j\ge 1$, be a kernel of the form
\begin{equation*}
\lre_{j-2n}^i(x,y)=
 \frac{\lre_{m}(x,y)}{\rho^{2k}(x,y)},
\end{equation*}
where $m-2k\ge j-2n$. We denote by $E_{j-2n}$ the corresponding
isotropic operator.

\begin{thrm}
\label{E1properties}
 Let $T$ be a smooth tangential vector field.
Set $E$ to be an operator with kernel of the form
 $\lre^i_{1-2n}(x,y)R_1(x)$ or $\lre^i_{2-2n}(x,y)$ .
Then we have the following properties:
\begin{align*}
&i)\ E_{j-2n}:L^p(D)\rightarrow L^s(D)\qquad j\ge 1\\
&ii)\ E:L^{\infty,2+\epsilon,0}(D)\rightarrow
 \Lambda_{\alpha,2-\epsilon'}(D)
\qquad 0<\epsilon,\epsilon', \quad \alpha+\epsilon+\epsilon'<1\\
&iii)\ \gamma^{\ast}TE:\Lambda_{\alpha,2+\epsilon}(D)
\rightarrow L^{\infty,\epsilon',0}(D) \qquad \epsilon<\epsilon'\\
&iv)\ E:L^{\infty,\epsilon,\delta}(D)\rightarrow
 L^{\infty,\epsilon',0}(D) \qquad \epsilon<\epsilon',
 \quad \delta<1/2 +(\epsilon'-\epsilon)/2
\end{align*}
for any $1\le p\le s\le\infty$ with $1/s>1/p-j/2n$.
\end{thrm}

\section{Integral representation}
 We give here the basic integral
representation.  As in \cite{EhLi}, the crucial property here is
that certain operators in the integral representation, while of
type 0, are of smooth type 1, hence multiplication by a factor of
$R_1$ turns them into type 1 operators with smoothing properties.

We start with the differential forms
\begin{align*}
&\beta(x,y)=
 \frac{\partial_{x}\rho^2(x,y)}{\rho^2(x,y)}\\
&\alpha_{\epsilon}(x,y)=\xi(x)\frac{\partial
r_{\epsilon}(x)}{\phi_{\epsilon}(x,y)},
\end{align*}
where $\xi(x)$ is a smooth patching function which is equivalently
1 for $|r(x)|<\delta$ and 0 for $|r(x)|>\frac{3}{2} \delta$, and
$\delta>0$ is sufficiently small.  We define
\begin{equation*}
C^{\epsilon}=C(\alpha_{\epsilon},\beta)
 = \sum_{\mu=0}^{n-2}
 \left(\frac{1}{2\pi i}\right)^n
C_{\mu}(\alpha_{\epsilon},\beta),
\end{equation*}
where
\begin{equation*}
C_{\mu}(\alpha_{\epsilon},\beta)
 = \alpha_{\epsilon}\wedge
 \beta\wedge(\mdbar_{x}\alpha_{\epsilon})^{\mu}\wedge
 (\mdbar_{x}\beta)^{n-\mu-2}.
\end{equation*}
Denoting the Hodge $\ast$-operator by $\ast$, we then define
\begin{equation*}
\lrl^{\epsilon}(x,y)=-\ast_{x}\overline{C^{\epsilon}(x,y)}.
\end{equation*}
We also write
\begin{equation*}
K^{\epsilon}(x,y) = \frac{1}{(2\pi
i)^n}\alpha_{\epsilon}\wedge(\mdbar_{x}\alpha_{\epsilon})^{n-1}
\end{equation*}
and
\begin{equation*}
\Gamma^{\epsilon}(x,y)
 =\frac{(n-2)!}{2\pi^n}\frac{1}{\rho^{2n-2}}.
\end{equation*}

The kernels in our integral representation are defined through the
following:
\begin{align*}
\lrt_0^{\epsilon}(x,y)&=\vartheta_{x}\lrl^{\epsilon}(x,y)-
 \ast\overline{K}^{\epsilon}(x,y)+
 \mdbar_{x}\Gamma^{\epsilon}(x,y)\\
\lrp_0^{\epsilon}(x,y)&=\ast\partial_{x}\overline{K^{\epsilon}(x,y)}.
\end{align*}
We denote the operators with kernels $\lrt_0^{\epsilon}$ and
$\lrp_0^{\epsilon}$
by ${\mathbf T_0^{\epsilon}}$ and ${\mathbf P_0^{\epsilon}}$,
respectively.  We also define ${\mathbf T_0}$ and ${\mathbf P_0}$
to be the operators with kernels which are the limits of the
kernels $\lrt_0^{\epsilon}$ and $\lrp_0^{\epsilon}$, respectively,
as $\epsilon\rightarrow 0$.
\begin{thrm}
 \label{basic0}
 Let $f\in L^2(D)\cap \mbox{Dom }\mdbar$.  Then (in the sense of
 distributions)
\begin{equation}
\label{basiceqn0}
 f={\mathbf T}_0\mdbar f + {\mathbf P}_0 f +
(A_{(0,2)}+E_{2-2n})\mdbar f + ( A_{(0,1)} +E_{1-2n})f.
\end{equation}
\end{thrm}
The proof follows as in \cite{Eh09a},
 up until the examination of the term (\ref{improvement}) below.
 By carefully considering such a term, we are able to show that
 the resulting integral kernels in the representation are of
 double type $(0,1)$, whereas the analysis in \cite{Eh09a} would
 only grant us kernels of double type $(-1,1)$.
This information is not needed in the $q>0$ case due to
  the difference in
the definitions of the term $P_{\epsilon}(\zeta,z)$ in the cases
$q=0$ and $q>0$.
\begin{proof}[Sketch of proof.]
 We first prove the corresponding statement for the domain $D_{\epsilon}$
 and in the case $f\in
 C^1(\overline{D})$.

 Our starting point is the Bochner-Martinelli-Koppelman (BMK)
formula for $f\in C^1(\overline{D}_{\epsilon})$.  Let $B_0$ be
defined by
\begin{equation*}
B_0=\Omega_0(\beta)=\frac{1}{(2\pi i)^n}
 \beta\wedge(\mdbar_{\zeta}\beta)^{n-1}.
\end{equation*}
Then for $y\in D_{\epsilon}$
\begin{multline}
\label{intrepwbndry}
 f(y)=\int_{\partial D_{\epsilon}}f(x)\wedge B_0(x,y)
  -\int_{D_{\epsilon}}\mdbar f(x)\wedge B_0(x,y)
  \\
  +(f(x),\lre_{1-2n}(x,y))+(\mdbar
  f(x),\lre_{2-2n}(x,y)).
 \end{multline}

Define the kernels $K_0^{\epsilon}(x,y)$ by
$K_0^{\epsilon}(x,y)=\Omega_0(\alpha_{\epsilon})$.  We then
proceed to replace the boundary integral in the BMK formula by
\begin{equation*}
\int_{\partial D_{\epsilon}}f\wedge K_0^{\epsilon}.
\end{equation*}
Let $x_0\in\partial D_{\epsilon}$ be a fixed point and $U$ a
sufficiently small neighborhood of $x_0$ on which $\zeta$ is a
local coordinate map. $F(\zeta,z)$ vanishes on the diagonal of
$\overline{U}\times\overline{U}$, so Hefer's theorem applies to
give us
\begin{equation*}
F(\zeta,z)=\sum_{j=1}^n h_j(\zeta,z)(\zeta_j-z_j).
\end{equation*}
We set
\begin{equation*}
 \alpha^0(\zeta,z)=\frac{\sum_{j=1}^n
h_j(\zeta,z)d\zeta_j}{F}.
\end{equation*}
With the metric given by
\begin{equation*}
ds^2=\sum g_{jk}(\zeta) dz_jd\overline{z}_k,
\end{equation*}
 we define
\begin{align*}
&b^0(\zeta,z)=\sum_{j,k=1}^n
g_{jk}(\overline{\zeta}_k-\overline{z}_k)d\zeta_j\\
 &R^2(\zeta,z)=
 \sum_{j,k=1}^n
 g_{jk}(\zeta_j-z_j)(\overline{\zeta}_k-\overline{z}_k)\\
&\beta^0(\zeta,z)=\frac{b^0(\zeta,z)}{R^2(\zeta,z)}.
\end{align*}
With use of the transition kernels $C$ defined above, we have via
Koppelman's homotopy formula
\begin{equation*}
\Omega_0(\beta^0)=\Omega_0(\alpha^0)
 -\mdbar_{\zeta}C(\alpha^0,\beta^0)
 .
\end{equation*}
On $(\partial D\cap U)\times U$ we have
\begin{equation*}
\Omega_0(\alpha^{\epsilon})=\Omega_0(\alpha_0)
 +\frac{\lre_1(\zeta,z)}{\phi_{\epsilon}(\zeta,z)^n}
\end{equation*}
and on $U\times U$ we have
\begin{equation*}
\Omega_0(\beta)=\frac{R^{2n}}{\rho^{2n}}\Omega_0(\beta^0)
 +\lre_{2-2n}.
\end{equation*}
Thus we write
\begin{align*}
\Omega_0(\beta)&=\frac{R^{2n}}{\rho^{2n}}\Omega_0(\beta^0)
 +\lre_{2-2n}\\
 &=\frac{R^{2n}}{\rho^{2n}}(\Omega_0(\beta^0)-\Omega_0(\alpha^0))
 +\Omega_0(\alpha^0)
 +\frac{\lre_{2n+1}}{\rho^{2n}}\Omega_0(\alpha^0)+\lre_{2-2n},
\end{align*}
by which it then follows from the homotopy formula and the
relations between $b^0$ and $\rho^2$ and $R^2$, exactly as it was
obtained in \cite{LiMi}, that we have
\begin{align}
\label{trans}
 \Omega_0(\beta)
 =&\Omega_0(\alpha_{\epsilon})-
 \mdbar_{\zeta}C^{\epsilon}+
(\Omega_0(\alpha^0)-\Omega_0(\alpha_{\epsilon})) +
\frac{\lre_{2n+1}}{\rho^{2n}}\Omega_0(\alpha^0)\\
\nonumber
 &+\lre_{2-2n}+\mdbar_{\zeta}\Bigg[\Bigg(
 C\left(\alpha^0,\frac{\partial\rho^2+\lre_2}{\rho^2}\right)
 -C(\alpha_{\epsilon},\beta)\Bigg)\\
 \nonumber
  &\qquad+\sum_{\mu=0}^{n-2}\frac{\lre_{3+2\mu}}{(\rho^2)^{1+\mu}}
  C_{\mu}\left(\alpha^0,\frac{\partial\rho^2+\lre_2}{\rho^2}\right)
  \Bigg]\\
  \nonumber
  &+\sum_{\mu=0}^{n-2}\frac{\lre_{4+2\mu}}{(\rho^2)^{2+\mu}}
  C_{\mu}\left(\alpha^0,\frac{\partial\rho^2+\lre_2}{\rho^2}\right)
  .
\end{align}
We now work with $\zeta\in\partial D_{\epsilon}$ so that
$F=\phi_{\epsilon}$.

We have
\begin{equation*}
\frac{\lre_{2n+1}}{\rho^{2n}}\Omega_0(\alpha^0)
 =R_1\frac{\lre_{2n+1}}{\phi_{\epsilon}^n\rho^{2n}}
\end{equation*}
and
\begin{equation*}
 \Omega_0(\alpha_{\epsilon})-\Omega_0(\alpha^0)
 =\frac{\lre_1}{\phi_{\epsilon}^n}.
\end{equation*}
Furthermore,
\begin{equation*}
C_{\mu}\left(\alpha^0,\frac{\partial\rho^2+\lre_2}{\rho^2}\right)
 =R_1\frac{\lre_1}{\phi_{\epsilon}^{1+\mu}(\rho^2)^{n-1-\mu}},
\end{equation*}
and thus
\begin{align*}
\mdbar_{\zeta}\Bigg(
 \sum_{\mu}&\frac{\lre_{3+2\mu}}{(\rho^2)^{1+\mu}}
  C_{\mu}\left(\alpha^0,\frac{\partial\rho^2+\lre_2}{\rho^2}\right)
 \Bigg)=
 \mdbar_{\zeta}\left(
  R_1\frac{\lre_{4+2\mu}}{\phi_{\epsilon}^{1+\mu}\rho^{2n}}
\right)\\
&=\frac{\lre_{4+2\mu}}{\phi_{\epsilon}^{1+\mu}\rho^{2n}}
 +R_1\frac{\lre_{3+2\mu}}{\phi_{\epsilon}^{1+\mu}\rho^{2n}}
 +R_1\frac{\lre_{5+2\mu}+\lre_{4+2\mu}\wedge\mdbar_{\zeta}r_{\epsilon}}{\phi_{\epsilon}^{2+\mu}\rho^{2n}}.
\end{align*}

We have
\begin{equation}
\label{improvement}
C_{\mu}\left(\alpha^0,\frac{\partial\rho^2+\lre_2}{\rho^2}\right)
 -C_{\mu}(\alpha_{\epsilon},\beta)=
\frac{\lre_2}{\phi_{\epsilon}^{1+\mu}(\rho^2)^{n-1-\mu}},
\end{equation}
 and by carefully handling this term we obtain a
slight improvement, with tangible benefits, over its treatment in
\cite{LiMi} and even in \cite{Eh09a}, where the additional
information has no effect.
 In our relating boundary integrals to volume integrals, with the
 use of Stoke's Theorem, we replace $\rho^2$ with $P_{\epsilon}$,
 as they are equivalent on $\partial D_{\epsilon}$, in the
 kernels.  We therefore, use here the relation
\begin{equation*}
\mdbar_{\zeta}\Bigg(C_{\mu}\left(\alpha^0,\frac{\partial\rho^2+\lre_2}{\rho^2}\right)
 -C_{\mu}(\alpha_{\epsilon},\beta)\Bigg)
  =
  \mdbar_{\zeta}\left(
  \frac{\lre_2}{\phi_{\epsilon}^{1+\mu}P_{\epsilon}^{n-1-\mu}}\right)
  +
  \frac{\lre_2r^{\ast}_{\epsilon}\wedge\mdbar r_{\epsilon}}
  {\phi_{\epsilon}^{1+\mu}P_{\epsilon}^{n-\mu}}
\end{equation*}
for $\zeta\in\partial D_{\epsilon}$.  The upshot is that the
integral kernels in our final expression which stem from the term
(\ref{improvement}) are of double type $(0,1)$.

(\ref{trans}) can thus be written
\begin{align*}
 \Omega_0(\beta)=&\Omega_0(\alpha_{\epsilon})
  -\mdbar_{\zeta}C(\alpha_{\epsilon},\beta)
  +\lre_{2-2n}\\
  &+\sum_{\mu=0}^{n-2}\mdbar_{\zeta}\left(
  \frac{\lre_{4+2\mu}}{\phi_{\epsilon}^{1+\mu}P_{\epsilon}^{n}}
  \right)
  +r_{\epsilon}^{\ast}
\sum_{\mu=0}^{n-2}\frac{\lre_{2+2\mu}\wedge\mdbar
r_{\epsilon}}{\phi_{\epsilon}^{1+\mu}P_{\epsilon}^{n}}
+\frac{\lre_1}{\phi_{\epsilon}^n}\\
&
  +R_1\sum_{\mu=0}^{n-2}\frac{\lre_{3+2\mu}}{\phi_{\epsilon}^{1+\mu}\rho^{2n}}
  +\sum_{\mu=0}^{n-2}\frac{\lre_{4+2\mu}}{\phi_{\epsilon}^{1+\mu}\rho^{2n}}
 +R_1\sum_{\mu=0}^{n-2}\frac{\lre_{5+2\mu}+
 \lre_{4+2\mu}\wedge\mdbar_{\zeta}r_{\epsilon}}{\phi_{\epsilon}^{2+\mu}\rho^{2n}},
 \end{align*}
again for $\zeta\in\partial D_{\epsilon}$.

 Thus, after integrating by parts we obtain
\begin{align*}
\int_{\partial D_{\epsilon}}
 f\wedge &B_0  =\\
 & \int_{\partial D_{\epsilon}} f\wedge\Omega_0(\alpha_{\epsilon})
   + \int_{\partial D_{\epsilon}} \mdbar f\wedge
  C^{\epsilon} + \int_{\partial
  D_{\epsilon}}f\wedge\lre_{2-2n}\\
 &+ \sum_{\mu=0}^{n-2}\int_{D_{\epsilon}}\mdbar_{\zeta}\left[f\wedge
  \left(\mdbar_{\zeta}\left(
  \frac{\lre_{4+2\mu}}{\phi_{\epsilon}^{1+\mu}P_{\epsilon}^{n}}
  \right)
  +r_{\epsilon}^{\ast}
\frac{\lre_{2+2\mu}\wedge\mdbar
r_{\epsilon}}{\phi_{\epsilon}^{1+\mu}P_{\epsilon}^{n}}
+\frac{\lre_1}{\phi_{\epsilon}^n}\right)\right]\\
&+\sum_{\mu=0}^{n-2}\int_{\partial D_{\epsilon}}f\wedge
 \left(R_1\frac{\lre_{3+2\mu}}{\phi_{\epsilon}^{1+\mu}\rho^{2n}}
  +\frac{\lre_{4+2\mu}}{\phi_{\epsilon}^{1+\mu}\rho^{2n}}
 +R_1\frac{\lre_{5+2\mu}}{\phi_{\epsilon}^{2+\mu}\rho^{2n}}\right),
\end{align*}
where we use an application of Stoke's Theorem in the second
integral on the right hand side.

We now replace all occurrences of $\rho^2$ in the denominators by
$P_{\epsilon}$, since the two are equal on $\partial
D_{\epsilon}$, in the remaining boundary integrals, and then
change those boundary integrals to volume integrals by Stoke's
Theorem:
\begin{align*}
\int_{\partial D_{\epsilon}}
 f\wedge &B_0  =\\
 & \int_{D_{\epsilon}} \mdbar f\wedge\Omega_0(\alpha_{\epsilon})
   +\int_{D_{\epsilon}} f\wedge
   \mdbar_{\zeta}\Omega_0(\alpha_{\epsilon})
   -\int_{D_{\epsilon}} \mdbar f\wedge
  \mdbar_{\zeta}C^{\epsilon}+ \int_{D_{\epsilon}}\mdbar f\wedge\lre_{2-2n}\\
   &
+ \int_{D_{\epsilon}}f\wedge\lre_{1-2n}+
\sum_{\mu=0}^{n-2}\int_{D_{\epsilon}}\mdbar f\wedge
  \Bigg(
  \frac{\lre_{3+2\mu}}{\phi_{\epsilon}^{1+\mu}P_{\epsilon}^{n}}
  +R_1\frac{\lre_{4+2\mu}}{\phi_{\epsilon}^{2+\mu}P_{\epsilon}^{n}}\\
 &\qquad
 +\frac{\lre_{5+2\mu}}{\phi_{\epsilon}^{1+\mu}P_{\epsilon}^{n+1}}
 +R_1
 \frac{\lre_{4+2\mu}r^{\ast}_{\epsilon}}{\phi_{\epsilon}^{1+\mu}P_{\epsilon}^{n+1}}
  +R_1
\frac{\lre_{2+2\mu}r_{\epsilon}^{\ast}}{\phi_{\epsilon}^{1+\mu}P_{\epsilon}^{n}}
+\frac{\lre_1}{\phi_{\epsilon}^n}\Bigg)\\
& + \sum_{\mu=0}^{n-2}\int_{D_{\epsilon}}f\wedge
  \Bigg(R_1
\frac{\lre_{1+2\mu}r_{\epsilon}^{\ast}}{\phi_{\epsilon}^{1+\mu}P_{\epsilon}^{n}}
+R_1
\frac{\lre_{3+2\mu}r_{\epsilon}^{\ast}}{\phi_{\epsilon}^{1+\mu}P_{\epsilon}^{n+1}}+
\frac{\lre_0}{\phi_{\epsilon}^n}+\frac{\lre_2+\lre_1\wedge\mdbar
r_{\epsilon}}
 {\phi_{\epsilon}^{n+1}}\Bigg)\\
 &+\sum_{\mu=0}^{n-2}\int_{D_{\epsilon}}f\wedge
 \Bigg(\frac{\lre_{3+2\mu}}{\phi_{\epsilon}^{1+\mu}P_{\epsilon}^n}
  +R_1\frac{\lre_{2+2\mu}}{\phi_{\epsilon}^{1+\mu}P_{\epsilon}^n}+
  R_2\frac{\lre_{3+2\mu}}{\phi_{\epsilon}^{2+\mu}P_{\epsilon}^n}\\
  &\qquad +R_1\frac{\lre_{4+2\mu}}{\phi_{\epsilon}^{1+\mu}P_{\epsilon}^{n+1}}
  +\frac{\lre_{5+2\mu}}{\phi_{\epsilon}^{1+\mu}P_{\epsilon}^{n+1}}
  \Bigg).
\end{align*}

Inserting this expression of the boundary integral into
(\ref{intrepwbndry}), and using our notation of operators of a
certain type, we can write
\begin{align*}
f(z)=&\int_{D_{\epsilon}}f\wedge
 \mdbar_{\zeta}\Omega_0(\alpha_{\epsilon}) +\int_{D_{\epsilon}}
 \mdbar f\wedge \Omega_0(\alpha_{\epsilon})
 -\int_{D_{\epsilon}}\mdbar f\wedge
 \mdbar_{\zeta} C^{\epsilon}\\
 &-\int_{D_{\epsilon}}\mdbar f\wedge B_0^{\epsilon}
  \\
   & +(f,\lre_{1-2n})
   +\left(f,\lra^{\epsilon}_{(0,1)}\right)
  +(\mdbar f,\lre_{2-2n})
  +\left(\mdbar f,\lra_{(0,2)}^{\epsilon}\right).
\end{align*}
Then as in \cite{LR86}, a rearrangement of terms leads to the
representation for $f\in C^1(\overline{D})$ and $z\in
D_{\epsilon}$
\begin{equation*}
 f={\mathbf T}_0^{\epsilon}\mdbar f + {\mathbf P}_0^{\epsilon} f +
(A_{(0,2)}^{\epsilon}+E_{2-2n})\mdbar f + ( A_{(0,1)}^{\epsilon}
+E_{1-2n})f.
\end{equation*}
We then let $\epsilon\rightarrow0$ and the right hand side
approaches the right hand side of (\ref{basiceqn0}).  This
establishes (\ref{basiceqn0}) for $f\in C^1(\overline{D})$.

To see the theorem is valid for
 $f\in L^2(D)\cap \mbox{Dom }\mdbar$, we refer to
Theorems \ref{dertype1}, \ref{a0a2}, and \ref{E1properties} to
show all operators in (\ref{basiceqn0}) map $L^p$ to $L^p$ for
$p\ge 2$.
\end{proof}

\section{A symmetry argument}
From \cite{EhLi} we have the
\begin{prop}
 \label{cancel0}
\begin{equation*}
{\mathbf P}_0 - {\mathbf P}_0^{\ast} =  A_{(0,1)}.
\end{equation*}
\end{prop}

We denote by $\lrs_{\infty}(x,y)$ a kernel which has the
properties that $\forall \lambda>0$
\begin{align*}
& \int_{D}|\lrs_{\infty}(x,y)|^{\lambda} dV(x)<\infty\\
& \int_{D}|\lrs_{\infty}(x,y)|^{\lambda} dV(y)<\infty,
\end{align*}
 and which is also smoothing in the sense that $\forall k$
\begin{equation*}
\| (\gamma T)^k{\mathbf S}_{\infty} f \|_{L^{\infty}}\lesssim
\|f\|_2,
\end{equation*}
where ${\mathbf S}_{\infty}$ is an operator associated with a
$\lrs_{\infty}$ kernel.  We have the mapping property
\begin{equation*}
{\mathbf S}_{\infty}:L^2(D)\rightarrow L^p(D)
\end{equation*}
$\forall p \ge 2$ by the generalized Young's inequality \cite{Ra}.

\begin{lemma}
\begin{equation*}
\mdbar_{x}  R_1\lrp_0^{\ast}=\mdbar_{x} \lra_{1}+\lrs_{\infty}.
\end{equation*}
\end{lemma}
\begin{proof}
$\lrp_0^{\ast}$ has the form
\begin{equation*}
\lrp_0(x,y)^{\ast}=R_2(y)\frac{l(y)}{\phi(y,x)^{n+1}}+\lra_{(0,1)},
\end{equation*}
where $l(y)$ is smooth and real-valued, and we use
$R_1(y)=R_1(x)+\lre_1$.  Thus
\begin{align*}
\mdbar_{x}R_1\lrp_0^{\ast}
 &=\mdbar_{x}\left(
  R_{3}(y)\frac{l(y)}{\phi(y,x)^{n+1}}\right)
  +\mdbar_{x}\lra_{1}\\
 &=\mdbar_{x}\lra_{1}+\lrs_{\infty},
\end{align*}
since
\begin{equation*}
\mdbar_{x}\left(
  R_{3}(y)\frac{l(y)}{\phi(y,x)^{n+1}}\right)=\lrs_{\infty},
\end{equation*}
due to the fact that $\phi(y,x)$ is holomorphic in
 the $x$ variable near $x=y$.
\end{proof}

We let $N$ be the \dbar-Neumann operator on $(0,1)$-forms, and we
make use of the relation
\begin{equation*}
B=I-\mdbar^{\ast}N\mdbar
\end{equation*}
to obtain
\begin{cor}
 \label{dbarp}
\begin{equation*}
{\mathbf P}_0^{\ast} R_1(f- Bf) = A_{1} (f-Bf) + {\mathbf
S}_{\infty} N \mdbar f.
\end{equation*}
\end{cor}
\begin{proof}
\begin{align*}
{\mathbf P}_0^{\ast} R_1(f-Bf) &= (f-Bf,R_1\lrp_0^{\ast})\\
&=(\mdbar^{\ast}N\mdbar f,R_1\lrp_0^{\ast})\\
&=(N\mdbar f,\mdbar_{\zeta}R_1\lrp_0^{\ast})\\
&=(N\mdbar f,\mdbar_{\zeta}\lra_{1}+\lrs_{\infty})\\
&=(\mdbar^{\ast}N\mdbar f,\lra_{1})+ {\mathbf S}_{\infty} N \mdbar
f.
\end{align*}
\end{proof}

\section{Estimates}

We apply Theorem \ref{basic0} to $f-Bf$.
\begin{equation*}
f-Bf={\mathbf T}_0\mdbar f + {\mathbf P}_0 (f-Bf) +
(A_{(0,2)}+E_{2-2n})\mdbar f + (A_{(0,1)} +E_{1-2n})(f-Bf),
\end{equation*}
and then we use Proposition \ref{cancel0} to obtain
\begin{equation*}
 f-Bf={\mathbf T}_0\mdbar f + {\mathbf P}_0^{\ast} (f-Bf) +
(A_{(0,2)}+E_{2-2n})\mdbar f + (A_{(0,1)} +E_{1-2n})(f-Bf).
\end{equation*}

Then after applying Corollary \ref{dbarp}, we write
\begin{thrm}
\begin{equation}
\label{fbf}
 f-Bf=({\mathbf T}_0+A_{(0,2)}+E_{2-2n})\mdbar f + {\mathbf
S}_{\infty} N\mdbar f+( A_{(0,1)} +E_{1-2n})(f-Bf).
\end{equation}
\end{thrm}

We introduce the notion of $Z$ operators.  We denote by $Z_1$
those operators which are of a type $\ge 1$ and $\epsilon=0$ (that
is, we consider operators on spaces of the domain, $D$) or are
$E_{1-2n}\circ R_1$ and also the sum of such operators.
 We further include in our notation of
$Z_1$ those operators whose kernels conform to our definition of
admissible with the factors, $R_1$ being replaced by a factor of
$\gamma$. Thus, for instance, a kernel of smooth type 1 multiplied
by $\gamma^2$ will be a $Z_1$ operator.  $E_{1-2n}\circ\gamma$
will also be a $Z_1$ operator.

Let $\Z_1(x,y)$ be a kernel of a $Z_1$ operator.  We define $Z_j$
operators to be those operators which have kernels of the form
$\Z_1(x,y)\lre_{j-1}(x,y)$.
 We establish
mapping properties for $Z_j$ operators.

   For the proof of the properties of $Z_j$ operators,
we will refer to a lemma of Schmalz (see the proof of Lemma 3.2 in
\cite{Sc89}) which provides a useful coordinate system in which to
prove estimates.
\begin{lemma}
 \label{schmalzlem}
Let $E_{\delta}(y):=\{x\in D: \rho(x,y) < \delta \gamma(y)\}$ for
$\delta>0$.  Then there is a constant $c$ and a number, $m\in
\{1,\ldots,2n\}$ such that for all $y\in D$, there are local
coordinates around each $x\in E_c(y)$ given by $\zeta_j =
t_j+it_{j+n}$ for $1\le j\le n$, and
 \begin{equation*}
 \{
-r(\zeta),t_1,\ldots,_{\hat{m}}\ldots,t_{2n}\},
\end{equation*}
where $t_m$ is omitted, forms a coordinate system near $x$. We
have the estimate
\begin{equation}
\label{volest}
 dV \lesssim \frac{1}{\gamma(z)}\left|
dr(\zeta)\wedge dt_1\wedge \ldots_{\hat{m}}\ldots\wedge
dt_{2n}\right|,
\end{equation}
where $dV$ is the Euclidean volume form on $\mathbb{R}^{2n}$.
\end{lemma}

\begin{lemma}
\label{weightedz}
 For $2\le p<\infty$ and $1\le j\le n+1$
\begin{equation}
\label{jnorm}
 \| Z_j f\|_{L^p}\lesssim \|  f\|_{L^{p,jp}}.
\end{equation}
Let $0<\epsilon'<\epsilon$.
 \begin{equation}
\label{inftynorm}
 \| Z_j f\|_{L^{\infty,\epsilon,0}}\lesssim \|
f\|_{L^{\infty,j+\epsilon',0}}.
\end{equation}
\end{lemma}
\begin{proof}
 We prove (\ref{jnorm}) for kernels of the form
 $\lra_{(1,1)}(x,y)\lre_{j-1}$.
 We show below that $\lra_{(1,1)}(x,y)$ satisfies
\begin{equation}
\label{aeproperty}
 \sup_{y\in\Omega}\int \frac{1}{\gamma^j(x)}
 |\lra_{(1,1)}(x,y)\lre_{j-1}(x,y)|
 |r(x)|^{-\delta}|r(y)|^{\delta}dV(x) <\infty
\end{equation}
for $\delta<1/2$.  The lemma then follows from the generalized
Young's inequality.

We further restrict our proof to the cases in which $\lra_{(1,1)}$
satisfies
\begin{align*}
& i)\ \frac{1}{\gamma(x)}|\lra_{(1,1)}| \lesssim
\gamma(x)\frac{1}{P^{n-1/2-\mu}|\phi|^{\mu+1}}
\qquad \mu\ge 1 \\
& ii) \frac{1}{\gamma(x)}|\lra_{(1,1)}| \lesssim
\frac{1}{P^{n-1-\mu}|\phi|^{\mu+1}} \qquad \mu\ge 1\\
& iii) \frac{1}{\gamma(x)}|\lra_{(1,1)}| \lesssim
\frac{1}{\gamma(x)}\frac{1}{P^{n-3/2-\mu}|\phi|^{\mu+1}} \qquad
\mu\ge 1.
\end{align*}
We will prove the more difficult case $iii)$, as cases $i)$ and
$ii)$ follow similar arguments, and we leave the details of those
cases to the reader.

We denote the critical points of $r$ by $p_1,\ldots,p_k$, and take
$\varepsilon$ small enough so that in each
\begin{equation*}
U_{2\varepsilon}(p_j)=\{x: D\cap \rho(x,p_j)<2\varepsilon\},
\end{equation*}
for $j=1,\ldots,k$,
 there are coordinates $u_{j_1},\ldots,u_{j_m},v_{j_{m+1}},\ldots,v_{j_{2n}}$ such
 that
 \begin{equation}
 \label{rcoor}
 -r(x)=u_{j_1}^2+\cdots+u_{j_m}^2-v_{j_{m+1}}^2-\cdots - v_{j_{2n}}^2,
 \end{equation}
with $u_{j_{\alpha}}(p_j)=v_{j_{\beta}}(p_j)=0$ for all $1\le
\alpha\le m$ and $m+1\le\beta\le 2n$, from the Morse Lemma.

 We
divide the estimates into subcases depending on whether $y\in
U_{\varepsilon}$.

Subcase $a)$.  Suppose $y\in U_{\varepsilon}(p_j)$.  We estimate
\begin{equation}
\label{zUjep} \int_{U_{2\varepsilon}(p_j)}
\frac{\rho^{j-1}(x,y)}{\gamma^j(x)|\phi|^{\mu+1}P^{n-3/2-\mu}|r(x)|^{\delta}}
dV(x)
\end{equation}
and
\begin{equation}
\label{zDminus}
 \int_{D_{\epsilon}\setminus U_{2\varepsilon}}
\frac{\rho^{j-1}(x,y)}{\gamma^j(x)|\phi|^{\mu+1}P^{n-3/2-\mu}|r(x)|^{\delta}}
dV(x).
\end{equation}

We break up the integral in (\ref{zUjep}) into integrals over
$E_c(y)$ and its complement, where $c$ is as in Lemma
\ref{schmalzlem}, and we choose $c$ small enough so that, in
$E_c(y)$, we have $\rho(x,y)\lesssim \gamma(x)$.

In the case $U_{2\varepsilon}(p_j)\cap E_c(y)$, we use a
coordinate system, $s=-r(\zeta)$, $t_1,\ldots,t_{2n-1}$, and
estimate
\begin{align}
\label{z1stcoor}
 \int_{ U_{2\varepsilon}(p_j)\cap E_c(y)}
 &\frac{1}{\gamma(x)|\phi|^{\mu+1}P^{n-3/2-\mu}|r(x)|^{\delta}}
dV(x)\\
\nonumber
 &\qquad
 \lesssim
 \int_{\mathbb{R}_+^2}\frac{t^{2n-2}}{\gamma(\zeta)\gamma(z)s^{\delta}(\theta+s+t^2)^{\mu+1}(s+t)^{2n-3-2\mu}}dsdt\\
 \nonumber
&\qquad
 \lesssim
 \int_{\mathbb{R}_+^2}\frac{t^{2\mu-1}}{s^{\delta}(\theta+s+t^2)^{\mu+1}}dsdt\\
\nonumber
 &\qquad\lesssim
 \int_{\mathbb{R}_+^2} \frac{1}{s^{\delta}(\theta^{1/2}+s^{1/2}+t)^{3}} dsdt\\
 \nonumber
 &\qquad\lesssim
 \int_0^{\infty} \frac{1}{s^{\delta}(\theta+s)} dsdt\\
 \nonumber
 &\qquad\lesssim \frac{1}{\theta^{\delta}},
\end{align}
where we use the notation
 $\mathbb{R}_+^j=\overbrace{\mathbb{R}_+\times\cdots\times\mathbb{R}_+}^{j\mbox{ times}}$.

We now estimate the integral
\begin{equation}
 \label{zujepset}
\int_{ U_{2\varepsilon}(p_j)\setminus E_c(y)}
\frac{\rho^{j-1}(x,y)}{\gamma^j(x)|\phi|^{\mu+1}P^{n-3/2-\mu}|r(x)|^{\delta}}
dV(x).
\end{equation}

We can assume, without loss of generality that in
$U_{2\varepsilon}(p_j)$ there are coordinate charts $\zeta$ and
$z$ for $x$ and $y$ respectively.
 We define $w_1,\ldots,w_{2n}$
 by
 \begin{equation*}
 w_{\alpha}=\begin{cases}
 u_{j_{\alpha}}\quad \mbox{for } 1\le
\alpha\le m\\
 v_{j_{\alpha}} \quad \mbox{for } m+1\le\alpha\le 2n,
\end{cases}
\end{equation*}
  and we let $t_1,\ldots,t_{2n}$ be defined by
$\zeta_{\alpha}=t_{\alpha}+it_{n+\alpha}$.  From the Morse Lemma,
the Jacobian of the transformation from coordinates
$t_1,\ldots,t_{2n}$ to $w_1,\ldots,w_{2n}$ is bounded from below
and above and thus we have
\begin{equation*}
|\zeta-z| \simeq |w(\zeta)-w(z)|
\end{equation*}
for $\zeta$, $z\in U_{2\varepsilon}(p_j)$.

 From (\ref{rcoor}) we have
$\gamma(z)\gtrsim |w(z)|$, and thus
\begin{align*}
|w(\zeta)-w(z)| & \simeq |\zeta-z|\\
&
\gtrsim \gamma(z)\\
&\gtrsim |w(z)|\\
&\ge |w(\zeta)| -|w(\zeta)-w(z)|,
\end{align*}
and we obtain
\begin{align*}
|w(\zeta)|&\lesssim |w(\zeta)-w(z)|\\
&\simeq|\zeta-z|.
\end{align*}
 Thus, $ |w(\zeta)|\lesssim|\zeta-z| $ and $|w(\zeta)|\lesssim
\gamma(\zeta)$ for $\zeta,z\in U_{2\varepsilon}(p_j)$.

We can therefore bound the integral in (\ref{zujepset}) by
\begin{align}
\nonumber
 \int_{U_{2\varepsilon}(p_j)\setminus E_c(z)}
 &\frac{|\zeta-z|^{j-1}}{\gamma^j(\zeta)|\phi|^{\mu+1}P^{n-3/2-\mu}|r(\zeta)|^{\delta}}
dV(\zeta)\\
&\qquad\lesssim
 \nonumber
 \int_{U_{2\varepsilon}(p_j)\setminus E_c(z)}
 \frac{1}{\gamma^j(\zeta)|\phi|^{1/2}|\zeta-z|^{2n-1-j}
 |r(\zeta)|^{\delta}}
dV(\zeta)\\
\nonumber
 &\qquad\lesssim
\int_V \frac{u^{m-1}v^{2n-m-1}}{
 (u+v)^{2n-1}(\theta+u^2+v^2)^{1/2}(u^2-v^2)^{\delta}} dudv\\
 \label{casedel}
 &\qquad\lesssim
 \int_V \frac{1}{u(\theta+u^2)^{1/2}(u^2-v^2)^{\delta}} dudv,
\end{align}
where $V$ is a bounded region.  We make the substitution
$v=\tilde{v}u$, since $v^2<u^2$, and write (\ref{casedel}) as
\begin{align*}
\int_0^{M}\frac{1}{u^{2\delta}(\theta+u^2)^{1/2}}du
  \int_0^1\frac{1}{(1-\tilde{v}^2)^{\delta}}d\tilde{v}
   &\lesssim\frac{1}{\theta^{\delta}}\int_0^M
    \frac{1}{u^{2\delta}(1+u^2)^{1/2}}du\\
   &\lesssim\frac{1}{\theta^{\delta}},
\end{align*}
where $M>0$ is some constant.  We have therefore bounded
(\ref{zUjep}), and we turn now to (\ref{zDminus}).

In $D\setminus U_{2\varepsilon}$ we have that $\rho(x,y)$ and
$\gamma(x)$ are bounded from below so
\begin{equation*}
\int_{D\setminus U_{2\varepsilon}}
\frac{\rho^{j-1}(x,y)}{\gamma^j(x)|\phi|^{\mu+1}P^{n-3/2-\mu}|r(x)|^{\delta}}
dV(x)
 \lesssim \int_{D\setminus U_{2\varepsilon}}
\frac{1}{|r(x)|^{\delta}} dV(x)\lesssim 1,
\end{equation*}
the last inequality following because in $D\setminus
U_{2\varepsilon}$ $r$ can be chosen as a coordinate since
$\gamma(x)$ is bounded from below.  This finishes subcase $a)$.

Case $b)$.  Suppose $y\notin U_{\varepsilon}$.  We divide $D$ into
the regions $D\cap E_c(y)$ and $D\setminus E_c(y)$.

In $D\cap E_c(y)$ the same coordinates and estimates work here as
in establishing the estimates for the integral in
(\ref{z1stcoor}).

In $D\setminus E_c(y)$ we have $\rho(x,y)\gtrsim \gamma(y)$, but
$\gamma(y)$ is bounded from below, since $y\notin
U_{\varepsilon}$.  We therefore have to estimate
\begin{equation*}
\int_{D}\frac{1}{\gamma^j(x)|r(x)|^{\delta}}dV(x),
\end{equation*}
which is easily done by working locally with the coordinates
$w_1,\ldots,w_{2n}$ above.

(\ref{inftynorm}) is proved similarly.
\end{proof}

Applying $R_1$ to (\ref{fbf}) gives us
\begin{equation}
\label{bergmanz}
 R_1(y)(f-Bf)=Z_1 \mdbar f+{\mathbf S}_{\infty}
N\mdbar f
 +Z_1(f-Bf).
\end{equation}
 An iteration leads to
\begin{prop}
\begin{equation}
\label{gkb}
 R_k(y) (f-Bf)=Z_1\mdbar f+{\mathbf S}_{\infty} N\mdbar
f+ \sum_{j=1}^k\sum_{i_1+\cdots+i_j=k}
 Z_{i_1}\circ\cdots\circ Z_{i_j}
(f-Bf).
\end{equation}
\end{prop}
\begin{proof}
The case $k=1$ is just Equation \ref{bergmanz}.  Then, given
(\ref{gkb}) for a particular $k$, we prove the corresponding
equation $R_{k+1}(y)(f-Bf)$ by multiplying (\ref{gkb}) by a factor
of $R_1$:
\begin{align*}
R_{k+1}(y) (f-Bf) &=
 Z_1\mdbar f+{\mathbf S}_{\infty} N\mdbar f +
 \sum_{j=1}^k\sum_{i_1+\cdots+i_j=k+1}
 Z_{i_1}\circ\cdots\circ Z_{i_j}(f-Bf)\\
 &\qquad
 + \sum_{j=1}^k\sum_{i_1+\cdots+i_j=k}
 Z_{i_1}\circ\cdots\circ Z_{i_j}R_1(x) (f-Bf)\\
&=Z_1\mdbar f+{\mathbf S}_{\infty} N\mdbar f +
 \sum_{j=1}^k\sum_{i_1+\cdots+i_j=k+1}
 Z_{i_1}\circ\cdots\circ Z_{i_j}(f-Bf)\\
 &\qquad
 + \sum_{j=1}^k\sum_{i_1+\cdots+i_j=k}
 Z_{i_1}\circ\cdots\circ Z_{i_j}\circ Z_{1} (f-Bf)\\
&=Z_1\mdbar f+{\mathbf S}_{\infty} N\mdbar f+
\sum_{j=1}^{k+1}\sum_{i_1+\cdots+i_j=k+1}
 Z_{i_1}\circ\cdots\circ Z_{i_j}
(f-Bf),
\end{align*}
where we use
\begin{equation*}
R_1(y) Z_{i}= Z_{i}\circ R_1(x)+Z_{i+1}
\end{equation*}
in the first step.  We also use the commutator relations below
(see Theorem \ref{commutator}) to establish that
 $Z_j\circ{\mathbf S}_{\infty}={\mathbf S}_{\infty}$.
\end{proof}

When dealing with derivatives in order to establish the
$C^k$-estimates we will need to know how certain vector fields
commute with our $Z_j$ operators.  We make the important remark
here that in the coordinate patch of a critical point, the smooth
tangential vector fields are not smooth combinations of
derivatives with respect to the coordinate system described in
Lemma \ref{weightedz} (see \ref{rcoor}).
 To deal with this difficulty, we define
$\lre_{j,k}(x,y)$, for $j\ge 0$, for those double forms on open
sets $U\subset X\times X$ such that $\lre_{j,k}$ is smooth on $U$
and satisfies
\begin{equation*}
\lre_{j}(x,y)\lesssim  \xi(x)\rho^j(x,y),
\end{equation*}
where $\xi(x)$ is a smooth function on $U$ satisfying
\begin{equation*}
|\gamma^kD_k\xi|\lesssim 1,
\end{equation*}
for $D_k$ a $k^{th}$ order differential operator.

 In a neighborhood of a critical point, the manner in
 which we defined smooth tangential vector fields leads to
  combinations of
derivatives with respect to the coordinates of (\ref{rcoor}) with
coefficients only in $C^0(\overline{D})$ due to factors of
$\gamma$ which occur in the denominators of such coefficients.
 In general a $k^{th}$ order derivative of such
coefficients is in $\lre_{0,-k}$. Thus, when integrating by parts,
special attention has to be paid to these non-smooth terms.  We
obtain the following theorem from \cite{Eh09b}.
\begin{thrm}
\label{commutator}
 Let $X$ be a smooth tangential vector
field. Then
\begin{equation*}
\gamma^{\ast}
 X^{y}Z_1 =-Z_1\tilde{X}^{x}\gamma+Z_1^{(0)} +\sum_{\nu=1}^l
Z_1^{(\nu)}  W_{\nu}^{x}\gamma ,
\end{equation*}
where $\tilde{X}$ is the adjoint of $X$, the $Z_1^{(\nu)}$ are
also $Z_1$ operators, and          $W^{x}\in T_{x}^{1,0}(\partial
D)\oplus T_{x}^{0,1}(\partial D)$.
\end{thrm}

We need the
\begin{prop}
 \label{mapping}
  Let $p\ge 2$, $s>p$ and $k(s)$ an integer which satisfies
  \begin{equation*}
 \frac{1}{s}>\frac{1}{p}-\frac{k(s)}{2n+2}.
\end{equation*}
There are $0<\epsilon,\epsilon'$ small enough such that
\begin{align*}
& i)\ \sum_{j=1}^{k}\sum_{i_1+\cdots+i_j=k}
Z_{i_1}\circ\cdots\circ Z_{i_j} : L^p(D)\rightarrow L^{s}(D)\\
& ii)\ \gamma T \sum_{j=1}^{n+2} \sum_{i_1+\cdots+i_j=n+2}
Z_{i_1}\circ\cdots\circ Z_{i_j}
:L^{\infty,n+2+\epsilon,0}(D)\rightarrow
L^{\infty,\epsilon',0}(D)\qquad \epsilon'<\epsilon\\
&iii)\ Z_j: L^{\infty,m+j+\epsilon,0}(D) \rightarrow
L^{\infty,m+\epsilon',0}(D)\qquad \epsilon'<\epsilon.
\end{align*}
\end{prop}
\begin{proof}
$i)$
 We use Theorem \ref{dertype1}
$ii)$ and Theorem \ref{E1properties} $ii)$ to conclude
\begin{equation*}
Z_{i_1}\circ\cdots\circ Z_{i_j}: L^p(D)\rightarrow L^s(D),
\end{equation*}
where
\begin{equation*}
\frac{1}{s}>\frac{1}{p}-\frac{i_1+\cdots+i_j}{2n+2}.
\end{equation*}

$ii)$.  We first consider the case of $i_1=\ldots=i_{n+2}=1$.  We
use parts of Theorems \ref{dertype1} and \ref{E1properties} and
part $iii)$ of the current theorem to show the composition of
operators maps $L^{\infty,n+2+\epsilon,0}$ into
$L^{\infty,\epsilon',0}$. We can see this by using the commutator
relations and considering the two compositions
 $Z_1\circ\gamma T A_1\circ Z_1$, and $Z_1\circ\gamma T E\circ Z_1$.
From Theorems \ref{dertype1} and \ref{E1properties} we can find
$\epsilon_1,\ldots,\epsilon_4$ such that
$\epsilon_{j+1}<\epsilon_j$ and in
 the first case we have
\begin{equation*}
\|Z_1\circ \gamma T A_1\circ Z_1
f\|_{L^{\infty,\epsilon_1,0}}\lesssim\|\gamma T A_1\circ Z_1
f\|_{L^{\infty,\epsilon_2,\delta}}\lesssim
 \|Z_1f\|_{L^{\infty,2+\epsilon_3,0}}\lesssim
 \|f\|_{L^{\infty,3+\epsilon_4,0}},
\end{equation*}
and, in the second,
\begin{equation*}
\|Z_1\circ \gamma T E\circ Z_1 f\|_{L^{\infty,\epsilon_1}}\lesssim
\|\gamma T E\circ Z_1 f\|_{L^{\infty,1+\epsilon_2,0}}\lesssim
 \|Z_1f\|_{\Lambda_{\alpha,3-\epsilon_3}}\lesssim \|f\|_{L^{\infty,3+\epsilon_4,0}},
\end{equation*}
where the second and third inequalities are proved in the same way
as Theorem \ref{E1properties} $ii)$ and $iii)$.

We now consider the case in which $\max\{i_1,\ldots,i_j\}\ge2$,
and we show
\begin{equation*}
\gamma^{\ast}T Z_j:L^{\infty,j+\epsilon,0}\rightarrow
L^{\infty,\epsilon'}
\end{equation*}
for $\epsilon<\epsilon'$

For $T=T^y$ a smooth first order tangential differential operator
on $D$, with respect to the $y$ variable, we have
\begin{align*}
T\gamma(y)&\lesssim 1\\
Tr&=0\\
Tr^{\ast}&=\lre_0r\\
TP&=\lre_1+\lre_0 rr^{\ast}\\
 T\phi&=R_1\lre_{0}+\lre_1.
\end{align*}
We can therefore write for $j\ge2$
\begin{equation}
\label{-2} \gamma^{\ast}T Z_j=
 Z_{j-2}\gamma^2+Z_{j-1}\gamma+Z_j,
\end{equation}
where we use the convention that $Z_0$ refers to an operator of
type $\ge0$.

Noting that, as a result of the commutator relations, and our
 definition of $Z_j$ we also have for $j\ge2$
\begin{equation}
 \label{commutej}
\gamma^{\ast}T Z_j=Z_{j-1}\circ\gamma+Z_j+Z_j\circ\gamma T.
\end{equation}

We again use the commutator relations, (\ref{commutej}), and Lemma
\ref{weightedz} to reduce the proof to the examination of
\begin{equation*}
Z_{i_1}\circ\cdots\circ\gamma TZ_{i_k}\circ\cdots\circ Z_{i_j}
\end{equation*}
where $i_k\ge 2$.  But from (\ref{-2}) we can write
\begin{equation*}
Z_{i_1}\circ\cdots\circ
(Z_{i_k-2}\circ\gamma^2+Z_{i_k-1}\circ\gamma+Z_{i_k})
 \circ\cdots\circ Z_{i_j},
\end{equation*}
and since $i_1+\cdots+i_{k-1}+i_{k+1}+\cdots+i_j=n-k+2$, this case
follows from Theorem \ref{a0a2}, Lemma \ref{weightedz} and
property $iii)$ of the current theorem.

$iii)$.  This is just a corollary of Lemma \ref{weightedz}.
\end{proof}

We first establish the $L^p$-estimates.  The next theorem was
established in \cite{EhLi} in the case of strictly pseudoconvex
domains in $\mathbb{C}^n$, (see Theorem 5.2 of \cite{EhLi}).
\begin{thrm}
\label{bergmanlp}  For $p\ge 2$ let $k=k(p)$ satisfy
\begin{equation*}
\frac{1}{p}>\frac{1}{2}-\frac{k}{2n+2}.
\end{equation*}
Then
\begin{equation*}
\|Bf\|_{L^{p,kp}}\lesssim \|f\|_{L^p}.
\end{equation*}
\end{thrm}
\begin{proof}
We start by multiplying (\ref{basic0}) by $R_k$ and subtracting
(\ref{gkb}):
\begin{align}
\label{last}
 R_k Bf=&A_0 f+{\mathbf S}_{\infty} N\mdbar f\\
 \nonumber
 &+
\sum_{j=1}^k\sum_{i_1+\cdots+i_j=k}
 Z_{i_1}\circ\cdots\circ Z_{i_j}
f  + \sum_{j=1}^k\sum_{i_1+\cdots+i_j=k}
 Z_{i_1}\circ\cdots\circ Z_{i_j}
Bf.
\end{align}

The $L^p$ norm of the first term on the right hand side is bounded
by $\|f\|_p$ by Theorem \ref{a0a2} $ii)$.

We use the fact that
\begin{equation*}
{\mathbf S}_{\infty}:L^2\rightarrow L^p
\end{equation*}
for any $p\ge2$ and that
\begin{equation*}
\|N\mdbar f\|_{L^2}\lesssim \|f\|_{L^2}\lesssim \|f\|_{L^p}
\end{equation*}
(see \cite{Hor65}) to handle the second term.

For the last two sums in (\ref{last}) we use Proposition
\ref{mapping} $i)$ to show
\begin{equation*}
Z_{i_1}\circ\cdots\circ Z_{i_j}:L^2\rightarrow L^{p},
\end{equation*}
where $p$ is given by
\begin{equation}
\label{rhs} \frac{1}{p}>\frac{1}{2}-\frac{k}{2n+2}.
\end{equation}

Hence, the $L^p$ norm of the last two terms in (\ref{last}) is
bounded by
\begin{equation*}
\|f\|_{L^2}+\|Bf\|_{L^2}\lesssim \|f\|_{L^2} \lesssim \|f\|_{L^p}
\end{equation*}
which we use to finish the proof.
\end{proof}
Note that the right hand side of (\ref{rhs}) is 0 for $k(p)=n+1$.
Thus, we have the
\begin{cor}
For $2\le p<\infty$ we have the weighted estimates
\begin{equation*}
\|Bf\|_{L^{p,(n+1)p}}\lesssim \|f\|_{L^p}.
\end{equation*}
\end{cor}

For the $C^k$-estimates we take $k=n+2$ in (\ref{gkb}) and use
 Proposition \ref{mapping} $i)$
above to establish
\begin{equation}
\label{k=0}
 \| \gamma^{n+2}(f-Bf)\|_{L^{\infty}} \lesssim
 \| \mdbar f\|_{L^{\infty}} +\|f\|_2.
 \end{equation}

   We use the notation $D^k$ to denote a $k$-th order differential
operator, which is a sum of terms which are composites of $k$
 vector fields.
We define
\begin{equation*}
Q_k(f)=\sum_{j=0}^k\|\gamma^jD^j\mdbar f\|_{\infty}+\|f\|_2.
\end{equation*}

    $T^k$ will be used for  a $k$-th order tangential differential
operator, which is a sum of terms which are composites of $k$
 tangential vector fields.
\begin{lemma}
\label{tanglemma}
 Let $T^k$ be a tangential operator of order
$k$. Let $\epsilon>0$, then
\begin{equation*}
\|\gamma^{(n+2)+\epsilon+k}T^k(f-Bf)\|_{L^{\infty}}\lesssim
Q_k(f).
\end{equation*}
\end{lemma}
\begin{proof}
The proof is by induction.  The first step, $k=0$ is contained in
(\ref{k=0}).

We start with (\ref{gkb}) with $k=n+2$:
\begin{equation*}
R_{n+2}(y) (f-Bf)=Z_1\mdbar f+{\mathbf S}_{\infty} N\mdbar f+
\sum_{j=1}^{n+1}\sum_{i_1+\cdots+i_j=n+2}
 Z_{r_{i_1}}\circ\cdots\circ Z_{r_{i_j}}
(f-Bf)
\end{equation*}
and apply $\gamma^{\epsilon}\gamma^k T^k$ to both sides.  After
employing the commutator relations in Theorem \ref{commutator} and
the induction step, we have
\begin{align}
 \label{induct}
\gamma^{\epsilon}R_{n+2}\gamma^{k}T^k(f-Bf)
 =&\gamma^{\epsilon}\sum_{j=0}^k(Z_1\gamma^jT^j)(\mdbar f)+
\gamma^{\epsilon}\lrs_{\infty}N\mdbar f\\
 \nonumber
 &+\gamma^{\epsilon}\sum_{m=1}^{k-1}\sum_{j=1}^{n+2}\sum_{i_1+\cdots+i_j=n+2}
 Z_{r_{i_1}}\circ\cdots\circ Z_{r_{i_j}}
\gamma^{m}T^{m}(f-Bf)\\
\nonumber
 &+
 \gamma^{\epsilon}\gamma T \sum_{j=1}^{n+2}\sum_{i_1+\cdots+i_j=n+2}
 Z_{r_{i_1}}\circ\cdots\circ Z_{r_{i_j}}
\gamma^{k-1}T^{k-1}(f-Bf).
\end{align}
The $L^{\infty}$ norms of the first three terms on the right of
(\ref{induct}) are bounded by $Q_k(f)$ by Proposition
\ref{mapping} $i)$ and the induction step.
 For the last term we use Proposition \ref{mapping} $ii)$ to show its
 $L^{\infty}$ norm is bounded by
\begin{equation*}
\|\gamma^{(n+2)+\epsilon'+k-1}T^{k-1}(f-Bf)\|_{\infty}\lesssim
Q_{k-1}(f),
\end{equation*}
for $\epsilon'<\epsilon$, by the induction hypothesis.
\end{proof}

In order to generalize Lemma \ref{tanglemma} to include
non-tangential operators, we express a normal derivative of a
component
 of a function, $f$, in terms of
tangential operators acting on $f$ and components of $\mdbar f$,

 We have the decomposition in the following form:
\begin{equation}
 \label{normaldecomp}
\gamma N f=\sum_{j}a_{j}\gamma T_jf +
    \sum_j b_{j}f+ \sum_j c_{j}\gamma(\mdbar f)_j
,
\end{equation}
where the coefficients $a_{j}$, $b_{j}$, and $c_{j}$ are all of
the form $\lre_{0,0}$.  The decomposition is well known in the
smooth case (see \cite{LiMi}) and to verify (\ref{normaldecomp})
in a neighborhood of $\gamma=0$, one may use the coordinates
$u_{j_1},\ldots,u_{j_m},v_{j_{m+1}},\ldots,v_{j_{2n}}$ as in
(\ref{rcoor}) above.

It is then straightforward how to generalize Lemma
\ref{tanglemma}.   Suppose $D^k$ is a $k^{th}$ order differential
operator which contains the normal field at least once.   In
$\gamma^kD^k$ we commute $\gamma N$ with terms of the form $\gamma
T$, where $T$ is tangential, and we consider the operator $D^k=
D^{k-1}\circ \gamma N$, where $D^{k-1}$ is of order $k-1$.  The
error terms due to the commutation involve differential operators
of order $\le k-1$.  From (\ref{normaldecomp}) we just have to
consider $D^{k-1}\gamma T f$ and $D^{k-1}\mdbar f$.  The last two
terms are bounded by $Q_{k-1}(f)$, and we repeat the process with
$D^{k-1}\gamma T f$, until we are left with $k$ tangential
operators for which we can apply Lemma \ref{tanglemma}.  We
thereby obtain the weighted $C^k$ estimates given in Theorem
\ref{ckintro}.
\begin{thrm}  For $\epsilon>0$
\begin{equation*}
\|\gamma^{(n+2)+\epsilon+k} (f-Bf)\|_{C^{k}}\lesssim Q_k(f).
\end{equation*}
\end{thrm}
As an immediate consequence we obtain weighted $C^k$ estimates for
the canonical solution, the solution of minimal $L^2$ norm,
 to the $\mdbar$-equation.
 Let $v$ be any solution to $\mdbar v=f$ with $L^2$ estimates, the
 existence being guaranteed by H\"{o}rmander's solution \cite{Hor65}.
Then $u=v-Bv$ is the canonical solution to $\mdbar u=f$, and from
above we obtain the following estimates.
\begin{cor}
The canonical solution to $\mdbar u=f$ satisfies
\begin{equation*}
\|\gamma^{(n+2)+\epsilon+k} u\|_{C^{k}}\lesssim
 \| \gamma^{k} f\|_{C^{k}} + \|f\|_2.
\end{equation*}
\end{cor}

\bibliographystyle{amsplain}

\end{document}